\documentclass[11pt]{amsart}
\usepackage[margin=1.4in]{geometry}

\usepackage{amssymb,amsmath,amsthm,amsrefs}
\usepackage{tikz}
\usetikzlibrary{patterns}
\usepackage{cancel}
\usepackage{mathtools,float}
\usepackage[T1]{fontenc}
\usepackage{enumerate}
\usepackage{todonotes}
\usepackage{ifthen}
\usepackage{graphicx}
\usepackage{subcaption}
\usepackage{xargs}

\newcommand{\shadetheboxesPM}[1]{
    \foreach \x/\y in {#1}
    \fill[pattern color = black!75, pattern=north east lines] (\x,\y) rectangle +(1,1);
}

\newcommand{\drawverticallines}[3]{
    \foreach \x in {#2}
    \draw[line width=#3] (\x+0.01,0.01) -- (\x+0.01,#1+0.99);
}

\newcommand{\drawhorizontallines}[3]{
    \foreach \y in {#2}
    \draw[line width=#3] (0.01,\y+0.01) -- (#1+0.99,\y+0.01);
}

\newcommand{\drawclpattern}[2]{
	\foreach[count=\x] \y in {#1}
	{
		\filldraw (\x,\y) circle (#2 pt);
	}
}

\newcommand{\drawspecialbox}[1]{
    \foreach \x/\y/\z/\w/\A in {#1}
    {
        \fill[color = white!100, opacity=1, rounded corners = 1.5pt] (\x+0.125,\y+0.125) rectangle (\z-0.125,\w-0.125);
        \draw[color = black, rounded corners = 1.5pt] (\x+0.125,\y+0.125) rectangle (\z-0.125,\w-0.125);
        \fill[black] (\x/2+\z/2,\y/2+\w/2) node {\A};
    }
}

\newcommand{\drawspecialboxlarge}[1]{
    \foreach \x/\y/\z/\w/\A in {#1}
    {
        \fill[color = white!100, opacity=1, rounded corners = 1.5pt] (\x+0.125,\y+0.125) rectangle (\z-0.125,\w-0.125);
        \draw[color = black, rounded corners = 1.5pt] (\x+0.125,\y+0.125) rectangle (\z-0.125,\w-0.125);
        \fill[black] (\x/2+\z/2,\y/2+\w/2) node {\Large \A};
    }
}

\newcommand{\drawsolidshadedbox}[1]{
    \foreach \x/\y/\z/\w/\A in {#1}
    {
        \fill[color = gray!50, opacity=1, rounded corners=1.5pt] (\x+0.125,\y+0.125) rectangle (\z-0.125,\w-0.125);
        \draw[color = black, rounded corners=1.5pt] (\x+0.125,\y+0.125) rectangle (\z-0.125,\w-0.125);
        \fill[black] (\x/2+\z/2,\y/2+\w/2) node {\A};
    }
}

\newcommand{\drawlabels}[1]{
	\foreach \x/\y/\lab in {#1}
	{
		\draw (\x + 0.5,\y + 0.5) node {\lab};
	}
}

\newcommandx{\patt}[9][4={},5={},6={},7={},8={},9=4]
{
	\scalebox{#1}
	{
		\begin{tikzpicture}[baseline=(current bounding box.center)]
			\useasboundingbox (0.0,-.3) rectangle (#2+1,#2+1.3);
			\shadetheboxesPM{#4}
			\draw (0.01,0.01) grid (#2+1-0.01,#2+1-0.01);

			\drawsolidshadedbox{#6}
			\drawspecialbox{#7}
			\drawspecialboxlarge{#5}
			\drawclpattern{#3}{#9}
			\drawlabels{#8}
		\end{tikzpicture}
	}
}

\newcommandx{\cpatt}[8][4={},5={},6={},7={},8={}]
{
	\scalebox{#1}
	{
		\begin{tikzpicture}[baseline=(current bounding box.center)]
			\useasboundingbox (0.0,-.3) rectangle (#2+1,#2+1.3);
			\shadetheboxesPM{#4}
			\draw (0.01,0.01) grid (#2+1-0.01,#2+1-0.01);

			\drawsolidshadedbox{#6}
			\drawspecialbox{#7}
			\drawspecialboxlarge{#5}
			\drawclpattern{#3}{4}

			\foreach \x/\y in {#8}
			{
				\draw[line width=1] (\x,\y) circle (7 pt);
			}
		\end{tikzpicture}
	}
}

\newcommandx{\metapatt}[8][6={},7={},8={}]
{
    \scalebox{#1}
    {
        \begin{tikzpicture}[baseline=(current bounding box.center)]
					\foreach \width/\height in {#2}
					{
						\useasboundingbox (0.0,-.3) rectangle (\width+1,\height+1.3);
            \shadetheboxesPM{#6}

            \foreach \pos/\type in {#4}
            {
                \ifthenelse{\equal{\type}{v}}
                {
                    \drawverticallines{\height}{\pos}{1.7pt}
                }
                {
								    \ifthenelse{\equal{\type}{d}}
                    {
                      \draw[densely dashed] (\pos,0) -- (\pos,\height+1);
                    }
										{
											\drawhorizontallines{\width}{\pos}{1.7pt}
										}
                }
            }

            \foreach \pos/\type in {#3}
            {
                \ifthenelse{\equal{\type}{v}}
                {
                    \drawverticallines{\height}{\pos}{0.6pt}
                }
                {
										\drawhorizontallines{\width}{\pos}{0.6pt}
                }
            }

            \drawsolidshadedbox{#8}
            \drawspecialbox{#7}

            \foreach \x/\y/\type in {#5}
            {
                \ifthenelse{\equal{\type}{a}}
                {
                    \draw (\x,\y) circle (6pt);
                    \filldraw (\x,\y) circle (3pt);
                }
                {
                    \filldraw (\x,\y) circle (4pt);
                }
            }
					}
        \end{tikzpicture}
    }
}

\newcommandx{\dpatt}[9][6={},7={},8={},9={}]
{
    \scalebox{#1}
    {
        \begin{tikzpicture}[baseline=(current bounding box.center)]
					\foreach \width/\height in {#2}
					{
						\useasboundingbox (0.0,-.3) rectangle (\width+1,\height+1.3);
            \shadetheboxesPM{#6}

            \foreach \pos/\type in {#4}
            {
                \ifthenelse{\equal{\type}{v}}
                {
                    \drawverticallines{\height}{\pos}{1.7pt}
                }
                {
								    \ifthenelse{\equal{\type}{d}}
                    {
                      \draw[densely dashed] (\pos,0) -- (\pos,\height+1);
                    }
										{
											\drawhorizontallines{\width}{\pos}{1.7pt}
										}
                }
            }

            \foreach \pos/\type in {#3}
            {
                \ifthenelse{\equal{\type}{v}}
                {
                    \drawverticallines{\height}{\pos}{0.6pt}
                }
                {
										\drawhorizontallines{\width}{\pos}{0.6pt}
                }
            }

            \drawsolidshadedbox{#8}
            \drawspecialbox{#7}

            \foreach \x/\y/\type in {#5}
            {
                \ifthenelse{\equal{\type}{a}}
                {
                    \draw9 (\x,\y) circle (6pt);
                    \filldraw (\x,\y) circle (3pt);
                }
                {
                    \filldraw (\x,\y) circle (4pt);
                }
            }

						\drawlabels{#9}
					}
        \end{tikzpicture}
    }
}

\pgfmathsetmacro{\patttablescale}{1.05}
\pgfmathsetmacro{\pattdispscale}{0.80}
\pgfmathsetmacro{\patttextscale}{0.5}

\usetikzlibrary{arrows,shapes,backgrounds,decorations,positioning,patterns}
\colorlet{lightgray}{black!15}

\theoremstyle{plain}
\newtheorem{theorem}{Theorem}[section]
\newtheorem{proposition}[theorem]{Proposition}
\newtheorem{corollary}[theorem]{Corollary}
\newtheorem{lemma}[theorem]{Lemma}
\newtheorem{conjecture}[theorem]{Conjecture}

\theoremstyle{definition}
\newtheorem{definition}[theorem]{Definition}
\newtheorem{remark}[theorem]{Remark}

\DeclareMathOperator{\BG}{bg}
\newcommand{\bg}[1]{\BG(#1)}
\newcommand{\bd}[1]{\partial(#1)}
\DeclareMathOperator{\DPERMS}{dperms}
\newcommand{\dperms}[1]{\DPERMS(#1)}
\DeclareMathOperator{\UPERMS}{uperms}
\newcommand{\uperms}[1]{\UPERMS(#1)}
\DeclareMathOperator{\UDPERMS}{udperms}
\newcommand{\udperms}[1]{\UDPERMS(#1)}

\DeclareMathOperator{\Av}{Av}
\newcommand{\av}[1]{\Av\!\left(#1\right)}
\DeclareMathOperator{\LRM}{lrm}
\newcommand{\lrm}[1]{\LRM(#1)}

\DeclareMathOperator{\STD}{st}
\newcommand{\std}[1]{\STD(#1)}
\DeclareMathOperator{\ENC}{enc}
\newcommand{\enc}[1]{\ENC(#1)}
\DeclareMathOperator{\BENC}{benc}
\newcommand{\benc}[1]{\BENC(#1)}

\newcommand*{\pw}{3 pt}
\newcommand{\drawclpatternPOINTS}[2]{
    \foreach[count=\x] \y in {#2}
    {
        \filldraw (\x,\y) circle (#1);
    }
}

\newcommand*{\picscale}{0.65}
\newcommand*{\picscales}{0.5}

\tikzstyle{mynode}=[circle, draw, thin,fill=gray!20, scale=0.8]

\newcommand*{\lw}{1 pt}
\newcommand*{\labelshift}{0.35}

\newcommand{\thepoints}[2][(0,0)]{
    \begin{scope}[shift={#1}]
        \foreach \i in {1,...,#2}{
            \filldraw (\i-1,1-\i) circle (\pw);
        }
    \end{scope}
}

\newcommand{\labelthepointsbelow}[2][(0,0)]{
    \begin{scope}[shift={#1}]
        \foreach[count=\i] \l in {#2}{
            \draw (\i-1-\labelshift,1-\i-\labelshift) node {$\l$};
        }
    \end{scope}
}

\newcommand{\labelthepointsabove}[2][(0,0)]{
    \begin{scope}[shift={#1}]
        \foreach[count=\i] \l in {#2}{
            \draw (\i-1+\labelshift,1-\i+\labelshift) node {$\l$};
        }
    \end{scope}
}

\newcommand{\thegrid}[2][1]{
    \pgfmathsetmacro\max{#1+#2}
    \foreach \i in {1,...,#2}{
        \draw[line width=\lw] (\i-1, \max-1) -- (\i-1, #2-\i) -- (\max-1, #2-\i);
    }
    \foreach \i in {1,...,#1}{
        \draw[line width=\lw] (0, \max-\i) -- (#2+\i-1, \max-\i) -- (#2+\i-1, 0);
    }

}

\newcommand{\someboxes}[1]{
	\foreach \i/\j in {#1}{
        \draw[line width=\lw] (\i, \j) -- (\i+1, \j) -- (\i+1, \j+1) -- (\i, \j+1) -- (\i,\j);
    }
}

\newcommand{\somepoints}[1]{
	\foreach \i/\j in {#1}{
        \filldraw (\i,\j) circle (\pw);
    }
}

\newcommand{\thegridext}[2][2]{
    \pgfmathsetmacro\max{#1+#2}
    \pgfmathsetmacro\maxes{#1-2}
    \foreach \i in {1,...,#2}{
        \draw[line width=\lw] (\i-1, \max-2) -- (\i-1, #2-\i) -- (\max-1, #2-\i);
    }
    \ifnum#1>2
    \foreach \i in {1,...,\maxes}{
        \draw[line width=\lw] (0, \max-\i-1) -- (#2+\i-1, \max-\i-1) -- (#2+\i-1, 0);
    }
    \fi
        \draw[line width=\lw] (0, #2) -- (\max-1, #2) -- (\max-1, 0);

    \draw[line width=\lw] (\max-2, #2) -- (\max-2, 0);
}

\newcommand{\thegridextf}[2][2]{
    \pgfmathsetmacro\max{#1+#2}
    \pgfmathsetmacro\maxes{#1-2}
    \foreach \i in {1,...,#2}{
        \draw[line width=\lw] (\i-1, \max-1) -- (\i-1, #2-\i) -- (\max-2, #2-\i);
    }
    \ifnum#1>2
    \foreach \i in {1,...,\maxes}{
        \draw[line width=\lw] (0, \max-\i-1) -- (#2+\i-1, \max-\i-1) -- (#2+\i-1, 0);
    }
    \fi
        \draw[line width=\lw] (0, #2) -- (\max-2, #2) -- (\max-2, 0);

    \draw[line width=\lw] (#2, \max-2) -- (#2, \max-1) -- (0, \max-1);
}

\newcommand{\permaslabel}[3]{
    \scalebox{#1}
    {
        \begin{tikzpicture}[baseline=(current bounding box.center)]
            \drawclpatternPOINTS{#2}{#3}
        \end{tikzpicture}
    }
}

\newcommand{\labeltheboxes}[2][(0,0)]{
    \begin{scope}[shift={#1}]
        \foreach \i/\j/\l/\xs/\ys in {#2}{
            \draw (\j-0.5+\xs,-\i+0.5+\ys) node {$\l$};
        }
    \end{scope}
}

\newcommand{\shadetheboxes}[2][(0,0)]{
\begin{scope}[shift={#1}]
    \foreach \i/\j in {#2}{
        \fill[pattern color = black!75, pattern=north east lines] (\j-1,-\i) rectangle (\j,-\i+1);
    }
\end{scope}
}

\newcommand*{\radius}{2 cm}

\newcommand{\thenodesA}[2][(0,0)]{
    \begin{scope}[shift={#1}]
    \tikzstyle{mynode}=[circle, draw, thin,fill=gray!20, scale=0.8]
    \foreach \i in {1,...,#2}{
        \foreach \j in {\i,...,#2}{
            \filldraw (\j-0.5,-\i-0.5) circle (\pw);
            \node[mynode] (\i \j) at (\j-0.5,-\i-0.5) {\tiny$\i\j$};
        }
    }
    \ifthenelse{#2>2}{
    \pgfmathtruncatemacro\rowstop{#2-2}
    \pgfmathtruncatemacro\columnstop{#2-1}
    \foreach \i in {1,...,\rowstop}{
        \pgfmathtruncatemacro\ii{\i+1}
        \foreach \j in {\ii,...,\columnstop}{
            \pgfmathtruncatemacro\jj{\j+1}
            \foreach \k in {\ii,...,\j}{
                \foreach \m in {\jj,...,#2}{
                    \path[thick, bend right] (\i \j) edge (\k \m);
                }
            }
        }
    }
    }{}
    \end{scope}
}

\newcommand{\thenodesAsmall}[2][(0,0)]{
    \begin{scope}[shift={#1}]
    \tikzstyle{mynode}=[circle, draw, thin,fill=gray!20, scale=0.4]
    \foreach \i in {1,...,#2}{
        \foreach \j in {\i,...,#2}{
            \filldraw (\j-0.5,-\i-0.5) circle (\pw);
            \node[mynode] (\i \j) at (\j-0.5,-\i-0.5) {\tiny$\i\j$};
        }
    }
    \ifthenelse{#2>2}{
    \pgfmathtruncatemacro\rowstop{#2-2}
    \pgfmathtruncatemacro\columnstop{#2-1}
    \foreach \i in {1,...,\rowstop}{
        \pgfmathtruncatemacro\ii{\i+1}
        \foreach \j in {\ii,...,\columnstop}{
            \pgfmathtruncatemacro\jj{\j+1}
            \foreach \k in {\ii,...,\j}{
                \foreach \m in {\jj,...,#2}{
                    \path[thick, bend right] (\i \j) edge (\k \m);
                }
            }
        }
    }
    }{}
    \end{scope}
}

\newcommand{\thesubnodesA}[2]{
    \tikzstyle{mynode}=[circle, draw, thin,fill=gray!20, scale=0.8]
    \foreach \i in {1,...,#1}{
        \foreach \j in {#1,...,#2}{
            \filldraw (\j-0.5,-\i-0.5) circle (\pw);
            \node[mynode] (\i \j) at (\j-0.5,-\i-0.5) {\tiny$\i\j$};
        }
    }
    \ifthenelse{#2>2}{
    \pgfmathtruncatemacro\rowstop{#1-1}
    \pgfmathtruncatemacro\columnstop{#2-1}
    \foreach \i in {1,...,\rowstop}{
        \pgfmathtruncatemacro\ii{\i+1}
        \foreach \j in {#1,...,\columnstop}{
            \pgfmathtruncatemacro\jj{\j+1}
            \foreach \k in {\ii,...,#1}{
                \foreach \m in {\jj,...,#2}{
                    \path[thick, bend right] (\i \j) edge (\k \m);
                }
            }
        }
    }
    }{}
}

\newcommand{\thecoreA}[2][(0,0)]{
    \begin{scope}[shift={#1}]
        \thepoints[(0,#2-1)]{#2}
        \thegrid{#2}
        \thenodesA[(0,#2+1)]{#2}
    \end{scope}
}

\newcommand{\thenodesB}[2][(0,0)]{
    \begin{scope}[shift={#1}]
    \tikzstyle{mynode}=[circle, draw, thin,fill=gray!20, scale=0.8]
    \foreach \i in {1,...,#2}{
        \foreach \j in {\i,...,#2}{
            \filldraw (\j-0.5,-\i-0.5) circle (\pw);
            \node[mynode] (\i \j) at (\j-0.5,-\i-0.5) {\tiny$\i\j$};
        }
    }
    \ifthenelse{#2>2}{
    \pgfmathtruncatemacro\rowstop{#2-2}
    \foreach \i in {1,...,\rowstop}{
        \pgfmathtruncatemacro\ii{\i+1}
        \pgfmathtruncatemacro\iii{\i+2}
        \foreach \j in {\iii,...,#2}{
            \pgfmathtruncatemacro\jj{\j+1}
            \pgfmathtruncatemacro\ju{\j-1}
            \foreach \k in {\ii,...,\ju}{
                \foreach \m in {\k,...,\ju}{
                    \path[thick, bend left] (\i \j) edge (\k \m);
                }
            }
        }
    }
    }{}
\end{scope}
}

\newcommand{\thenodesBsmall}[2][(0,0)]{
    \begin{scope}[shift={#1}]
    \tikzstyle{mynode}=[circle, draw, thin,fill=gray!20, scale=0.4]
    \foreach \i in {1,...,#2}{
        \foreach \j in {\i,...,#2}{
            \filldraw (\j-0.5,-\i-0.5) circle (\pw);
            \node[mynode] (\i \j) at (\j-0.5,-\i-0.5) {\tiny$\i\j$};
        }
    }
    \ifthenelse{#2>2}{
    \pgfmathtruncatemacro\rowstop{#2-2}
    \foreach \i in {1,...,\rowstop}{
        \pgfmathtruncatemacro\ii{\i+1}
        \pgfmathtruncatemacro\iii{\i+2}
        \foreach \j in {\iii,...,#2}{
            \pgfmathtruncatemacro\jj{\j+1}
            \pgfmathtruncatemacro\ju{\j-1}
            \foreach \k in {\ii,...,\ju}{
                \foreach \m in {\k,...,\ju}{
                    \path[thick, bend left] (\i \j) edge (\k \m);
                }
            }
        }
    }
    }{}
\end{scope}
}

\newcommand{\thesubnodesB}[2]{
    \tikzstyle{mynode}=[circle, draw, thin,fill=gray!20, scale=0.8]
    \foreach \i in {1,...,#1}{
        \foreach \j in {#1,...,#2}{
            \filldraw (\j-0.5,-\i-0.5) circle (\pw);
            \node[mynode] (\i \j) at (\j-0.5,-\i-0.5) {\tiny$\i\j$};
        }
    }
    \ifthenelse{#2>2}{
    \pgfmathtruncatemacro\rowstop{#1-1}
    \pgfmathtruncatemacro\columnbegin{#1+1}
    \foreach \i in {1,...,\rowstop}{
        \pgfmathtruncatemacro\ii{\i+1}
        \foreach \j in {\columnbegin,...,#2}{
            \pgfmathtruncatemacro\columnstop{\j-1}
            \foreach \k in {\ii,...,#1}{
                \foreach \m in {#1,...,\columnstop}{
                    \path[thick, bend left] (\i \j) edge (\k \m);
                }
            }
        }
    }
    }{}
}

\newcommand{\thecoreB}[2][(0,0)]{
    \begin{scope}[shift={#1}]
        \thepoints[(0,#2-1)]{#2}
        \thegrid{#2}
        \thenodesB[(0,#2+1)]{#2}
    \end{scope}
}

\newcommand{\completegraph}[2][(0,0)]{
    \begin{scope}[shift={#1}]
    \foreach \s in {1,...,#2}{
        \node[draw, circle] (\s) at ({-360/#2 * (\s - 1)}:\radius) {\tiny$\s$}; % If you want to label the nodes use \tiny$\s$
    }
    \foreach \s in {1,...,#2}{
        \foreach \t in {\s,...,#2}{
            \pgfmathsetmacro\tt{\t-1}
            {\ifnum\s=\t
            {}
            \else
            \draw (\t) -- (\s); %node[fill=white, near end] {\tiny$\s \pgfmathprintnumber\tt$};
            \fi
            }
        }
    }
    \end{scope}
}
\newcommand{\setpartitions}[2][(0,0)]{
    \begin{scope}[shift={#1}]
        \foreach \s in {1,...,#2}{
            \node (\s) at (\s-1,0) {\tiny$\s$};
        }
        \foreach \s in {1,...,#2}{
            \foreach \t in {\s,...,#2}{
            {\ifnum\s=\t
            {}
            \else
            \draw (\s) to [out=90, in=90] (\t);
            \fi
            }
            }
        }
    \end{scope}
}

\title[Pattern avoiding permutations and independent sets]
{Pattern avoiding permutations and independent sets in graphs}

\author[Christian Bean]{Christian Bean$^{\star}$}
\address{School of Computer Science, Reykjavik University, Reykjavik, Iceland}
\email{christianbean@ru.is}
\author[Murray Tannock]{Murray Tannock$^{\star}$}
\address{School of Computer Science, Reykjavik University, Reykjavik, Iceland}
\email{murray14@ru.is}
\author[Henning Ulfarsson]{Henning Ulfarsson$^{\star}$}
\address{School of Computer Science, Reykjavik University, Reykjavik, Iceland}
\email{henningu@ru.is}

\thanks{$^{\star}$ Research partially supported by grant 141761-051 from the
Icelandic Research Fund.}


\begin{document}

\begin{abstract}
	We introduce a new method for encoding permutations
	as weighted independent sets in a family of graphs we call cores. The
	encoding allows us to enumerate $(1324, 2143)$-, $(1234, 1324, 2143)$-,
	$(1234, 1324, 1432, 3214)$-avoiding permutations with respect to
	their number of ``boundary points'' and the size of the independent set
	in the graph they correspond to.
\noindent \\

\end{abstract}

\maketitle
\thispagestyle{empty}

\section{Introduction}
%!TEX root = patt-gons.tex

The \emph{standardization} of a string $s$ of distinct integers is the permutation $\std{s}$
obtained by replacing the $i$th smallest entry of $s$ with
$i$. A permutation $\pi$ of length $n$ \emph{contains} permutation $p$ of
length $k$ if there is a subsequence (not necessarily consisting of consecutive
entries) $\pi_{i_1} \pi_{i_2} \dots \pi_{i_k}$ whose standardization is $p$. In
this context, $p$ is called a (\emph{classical permutation}) \emph{pattern}.
The subsequence in $\pi$ is called an
\emph{occurrence} of $p$. If no occurrence exists then $\pi$ \emph{avoids} $p$.
Take for example $\pi = 51324$, which contains $p = 123$ as the subsequences
$134$ and $124$, but avoids $p = 231$.

Given a pattern $p$ we define $\Av_n(p)$ as the set of permutations of length
$n$ that avoid $p$, and $\av{p} = \cup_{n \geq 0} \Av_n(p)$. For a set $P$ of
patterns we let $\Av_n(P) = \cap_{p \in P} \Av_n(p)$, and
$\av{P} = \cup_{n \geq 0} \Av_n(P)$. A \emph{permutation class} is any set
of permutations defined by the avoidance of a set of classical patterns.

We illustrate a connection between independent sets in graphs to certain pattern avoiding
permutations. The classical problem of enumerating
non-crossing subgraphs in a complete graph drawn on a regular polygon will turn out to be
equivalent to a instance of $132$-avoiding permutations. The
methods developed can be extended to enumerate subclasses of permutations
avoiding the pattern $1324$, the only permutation class avoiding a single length
$4$ pattern that remains unenumerated.

\section{Encoding permutations as weights on a grid}
\label{sec:encoding_permutations_as_weights_on_a_grid}
%!TEX root = patt-gons.tex

A letter $\pi_i$ in a permutation $\pi$ is called a \emph{left-to-right
minimum} if $\pi_j > \pi_i$ for all $j = 1, \dots, i-1$. The letter $\pi_1$ is
always a left-to-right minimum in a non-empty permutation.
\emph{Left-to-right maxima}, \emph{right-to-left-minima} and
\emph{right-to-left-maxima} are defined analogously.

The sequence of left-to-right minima of a permutation $\pi$ will be called the
\emph{lrm-boundary} of the permutation and denoted $\lrm{\pi}$, e.g.,~if $\pi =
845367912$ then $\lrm{\pi} = 8431$. Given any permutation with $n$ left-to-right-minima, we can arrange the
lrm-boundary on a northwest-southeast diagonal and insert the remaining points of the
permutation in a \emph{staircase} (\emph{grid}), $B_n$, above this diagonal.
This grid is a collection of boxes with integer coordinates, labeled in matrix
notation, so the the top row has cells $(1 1), (1 2), \dots, (1 n)$, from left
to right, the next row has cells $(2 2), (2 3), \dots, (2 n)$, and so on.
Note that the lrm-boundary is a part of the grid.
See Figure~\eqref{fig:lrmgrid-132} for an example with the permutation $845367912$.
\begin{figure}[ht]
	\centering
	\begin{subfigure}[b]{0.3\textwidth}
		\begin{center}
		\begin{tikzpicture}[scale=0.3]
    		\drawclpatternPOINTS{\pw}{8,4,5,3,6,7,9,1,2}
		\end{tikzpicture}
		\end{center}
		\caption{The plot of the points $(i, \pi_i)$ of the permutation on the
		\mbox{Cartesian} plane.}
		\label{fig:permutation}
	\end{subfigure}
	\quad
	\begin{subfigure}[b]{0.3\textwidth}
		\begin{center}
		\begin{tikzpicture}[scale=\picscales]
			\thepoints[(0,3)]{4}
        	\labelthepointsbelow[(0,3)]{8,4,3,1}
        	\labeltheboxes[(0,4)]{1/3/\permaslabel{0.15}{0.25}{1}/0.15/0,
                       2/2/\permaslabel{0.15}{0.25}{1}/0/-0.25,
                       2/3/\permaslabel{0.15}{0.25}{1,2}/-0.15/0,
                       4/4/\permaslabel{0.15}{0.25}{1}/0/0}
        	\thegrid{4}
		\end{tikzpicture}
		\end{center}
		\caption{The permutation drawn on a staircase, $B_4$. The
		lrm-boundary is drawn on a diagonal.}
		\label{fig:lrmgrid-132}
	\end{subfigure}
	\quad
	\begin{subfigure}[b]{0.3\textwidth}
		\begin{center}
		\begin{tikzpicture}[scale=\picscales]
			\labelthepointsbelow[(0,0)]{\phantom{1}}
	        \thepoints[(0,3)]{4}
    	    \labeltheboxes[(0,4)]{1/3/1/0/0,
                       2/2/1/0/0,
                       2/3/2/0/0,
                       4/4/1/0/0}
        	\thegrid{4}
		\end{tikzpicture}
		\end{center}
		\caption{The permutation encoded as the number of points in
		each box in the staircase.}
		\label{fig:encoding-132}
	\end{subfigure}
	\caption{Three representations of the permutation $\pi = 845367912$.}
\end{figure}

We can give a coarser representation of this permutation by recording the number
of points in each box of the grid, see Figure~\eqref{fig:encoding-132}.
This is the (\emph{staircase}) \emph{encoding} of the
permutation, denoted $\enc{\pi}$. It is not necessary to label the left-to-right minima,
their values can be inferred from the number of points in each box. There are
other permutations (of length $8$) that have the encoding in
Figure~\eqref{fig:encoding-132}, e.g.,~$846379512$.

%!TEX root = patt-gons.tex

For permutations avoiding $132$, the presence of points in a box may constrain
other boxes to be empty. For example, in the encoding in
Figure~\eqref{fig:encoding-132}, the third box in row $1$ is occupied, which
implies the right-most boxes in rows $2$ and $3$ must be empty. These
constraints are symmetric. To capture these constraints we create a graph by
placing a vertex for every box and an edge between boxes that exclude one
another. More formally:

\begin{definition}\label{def:132core}
  Let $n \geq 0$ be an integer. The \emph{$132$-core of size $n$} is the
  labeled, undirected graph $D_n$ with vertex set $\{ (ij) \colon i = 1, \dots,
  n,\, j = i, \ldots, n \}$ and an edge between $(ij)$ and $(k\ell)$ if $i < k$,
  $j < \ell$ and the rectangle $\{i, i+1, \ldots, k\} \times \{j, j+1, \ldots,
  \ell\}$ is a subset of the vertex set.
\end{definition}

If $n = 0$, the $132$-core is the empty graph. See Figure~\ref{fig:cellincore} (left) for the
set of neighbours of a cell in a $132$-core, and Figure~\ref{fig:132-cores} for
the $132$-cores of sizes $n = 1, \dots, 5$. We use the letter ``$D$'' in
anticipation of Definition~\ref{def:cores}. Until we define another type of core
in Section~\ref{sec:permutations_avoiding_123} we will refer to $132$-cores as
\emph{cores}.

\begin{figure}[ht]
	\centering
	\begin{tikzpicture}[scale=\picscales]

			\thepoints[(0,7)]{8}
        	\thegrid{8}
        	\node[mynode] (3 5) at (5-0.5,6-0.5) {\tiny$3 5$};
        	\shadetheboxes[(0,8)]{1/3,1/4,
			                      2/3,2/4,
								  4/6,4/7,4/8,
								  5/6,5/7,5/8}

			\begin{scope}[shift={(11,0)}]
			\thepoints[(0,7)]{8}
        	\thegrid{8}
        	\node[mynode] (3 5) at (5-0.5,6-0.5) {\tiny$3 5$};
        	\shadetheboxes[(0,8)]{1/6,1/7,1/8,
			                      2/6,2/7,2/8,
								  4/4}
			\end{scope}

	\end{tikzpicture}
	\caption{The cell $(35)$ with its set of neighbours shaded; on the
	left in the $132$-core (Definition~\ref{def:132core}),
	and on the right in the $123$-core (Definition~\ref{def:123core}).}
	\label{fig:cellincore}
\end{figure}

\begin{figure}[hb]
\begin{center}
\begin{tikzpicture}[scale=\picscale]

    \thecoreA{1}

    \thecoreA[(2.5,-1)]{2}

    \thecoreA[(6,-2)]{3}

    \thecoreA[(10.5,-3)]{4}

    \thecoreA[(16,-4)]{5}

\end{tikzpicture}

\caption{The $132$-cores of sizes $1, \dots, 5$.}
\label{fig:132-cores}
\end{center}
\end{figure}

\begin{lemma}\label{lem:weighted_independent_sets}
  The staircase encoding, restricted to the set of permutation avoiding $132$ with
  exactly $n$ left-to-right minima, is a bijection to the set of
  weighted independent set of $D_n$.
\end{lemma}

\begin{proof}
    A rectangular region in a $132$-avoiding permutation drawn on a staircase
    grid (as in Figure~\eqref{fig:lrmgrid-132}) is increasing.
    Therefore, from the definition of the $132$-cores, the
    staircase encoding, restricted to the set of $132$-avoiding permutations with
    $n$ left-to-right minima, is a bijection to
    the set of weighted independent sets of $D_n$.
\end{proof}

For a concrete example, consider the permutation $845367912$ from
Figure~\eqref{fig:lrmgrid-132}. We can encode it as the independent set $\{(13),
(22), (23), (44)\}$ in $D_4$ along with the sequence of positive integer weights
$(1,1,2,1)$ recording the number of points at each vertex in the independent
set.

From the definition, it is easy to see that $D_n$ has $1 + 2 +
\dots + n = \binom{n+1}{2}$ vertices. To compute the number of edges consider a
vertex $(1j)$ in the first row of $D_n$. It has neighbours in the rectangular
region $\{2, \dots, j\} \times \{j+1, \dots, n\}$, giving a total of
$(j-1)(n-j) = (n+1)j - j^2 - n$ edges. The first row therefore contributes
  \[
    (n+1)\binom{n+1}{2} - \frac{2n+1}{3}\binom{n+1}{2} - n^2
    = \frac{n+2}{3}\binom{n+1}{2} - n^2
    = \binom{n}{3}
  \]
edges. If we only consider edges going to the southeast then row $i$ in the core
of size $n$ will look precisely like the top row in a smaller core of size
$n-i+1$. The total number of edges in $D_n$ is therefore $\binom{n+1}{4}$. The
enumerations of vertices and edges in the cores are special cases of the
following result.

\begin{proposition} \label{prop:cliquecount132}
    The number of cliques of size $k$ in $D_n$ is $\binom{n+1}{2k}$.
\end{proposition}

\begin{proof}
    The case of $k = 1$ (vertices) and $k = 2$ are given above. A clique of size
    $k > 2$ in $D_n$ either has:
    \begin{itemize}
        \item no vertex in the first row of the core, or
        \item has exactly one vertex in the first row.
    \end{itemize}
    In the first case, we have $\binom{n}{2k}$ cliques of size $k$, by
    induction. In the second case, assume the vertex in the first row is
    $(1j)$. The remaining vertices in the clique must form a clique of size
    $k-1$ in the rectangular region $\{2, \dots, j\} \times \{j+1, \dots, n\}$.
    Such a clique is obtained by choosing rows $r_1 < r_2 < \dots < r_{k-1}$, and
    independently columns $c_1 < c_2 < \dots < c_{k-1}$. Such a choice corresponds
    to the clique $(r_1, c_1), (r_2, c_2), \dots, (r_{k-1}, c_{k-1})$. The number of
    these choices is $\binom{j-1}{k-1}\binom{n-j}{k-1}$. Thus the number of
    cliques is
    \[
        \binom{n}{2k} + \sum_{j = k}^{n-k+1} \binom{j-1}{k-1}\binom{n-j}{k-1}.
    \]
    The sum is easily seen to be $\binom{n}{2k-1}$, which completes the proof.
\end{proof}

Because of Lemma~\ref{lem:weighted_independent_sets} the number of independent
sets in the core is more relevant for our purposes. In the following theorem
we use a multivariate generating function to enumerate the number of independent
sets of a given size. The variable $x$ tracks the number of boundary
points (which are left-to-right minima until we consider other boundaries) and $y$
tracks the size of the independent set. When we modify the generating function to
enumerate permutations we let $z$ track the number of points in the permutation and
$t$ track the number of boundary points.

\begin{theorem} \label{thm:indep_count_132}
  The number of independent sets of size $k$ in the $132$-core of size $n$ is
  given by the coefficient of $x^ny^k$ in the generating function $F = F(x,y)$
  satisfying the functional equation
  \[
    F = 1 + xF +
    \frac{x y F^2}{1 - y (F - 1)}.
  \]
  Solving this equation gives,
   \[
    F = \frac{(1-x)(1+y) + y - \sqrt{x(1+y)(xy+x-4y-2)+1}}{2y}.
  \]
\end{theorem}

\begin{proof}
If $n = 0$ the core is empty and the empty set is the only independent set.
This gives the term $x^0 y^0 = 1$ in Equation~\eqref{eqn:weighted_independent_sets}
below. If $n > 0$ then we
consider two cases: First, the case when the independent set has no vertex in
the top row of the core. This type of independent set corresponds to an independent
set in the core of size $n-1$ which is isomorphic to the subgraph induced by
the vertices in rows $2, \dots, n$. This gives the term $xF$ in the equation
below. Finally, we consider the case when the independent set has vertices in
the top row. There are no edges between vertices in the top row, therefore
any subset of vertices $\{1x_1,1x_2,\dots,1x_d\}$ (with $x_i < x_j$ if $i < j$)
can be in the independent set. The neighbours of $(1 x_1)$ form a
rectangular region $\{2,\dots,x_1\} \times \{x_1+1, \dots, n\}$ in rows
$2, \dots, x_1$. The remaining vertices in these rows form a core of size
$x_1 - 1$. This can be visualized on the staircase grid, as in
Figure~\ref{fig:132staircase_grid_with_shading}. Similarly, for $1 < i \leq d$,
the neighbours of $(1 x_i)$ form a rectangular region
$\{2, \dots, x_i\} \times \{x_i+1, \dots, n\}$ in rows $2, \dots, x_i$. The
remaining vertices in rows $x_{i-1}+1, \dots, x_i$ form a core of size
$x_i - x_{i-1}$. Lastly, the cells in rows $x_{d+1}, \dots, n$ do not share
an edge with a vertex in the top row, and form a core of
size $n-d$. This case gives the rest of the terms in the equation below.
\begin{equation}\label{eqn:weighted_independent_sets}
\begin{aligned}
    F &= 1 + xF + xyF^2 + \dots + xy^nF^2(F-1)^{n-1} +
                          \dotsm \\ &= 1 + xF + \frac{xyF^2}{1-y(F-1)}.
\end{aligned}\qedhere
\end{equation}
\end{proof}

\begin{figure}[h]
\begin{center}
\begin{tikzpicture}[scale=\picscales]
   \thepoints[(0,8)]{9}
   \labeltheboxes[(0,9)]{1/3/\permaslabel{0.2}{0.25}{1}/0/0,
                  1/6/\permaslabel{0.2}{0.25}{1}/0/0,
                  1/8/\permaslabel{0.2}{0.25}{1}/0/0}
  \shadetheboxes[(0,9)]{1/1,1/2,1/4,1/5,1/7,1/9,
                 2/4,2/5,2/6,2/7,2/8,2/9,
                 3/4,3/5,3/6,3/7,3/8,3/9,
                 4/7,4/8,4/9,
                 5/7,5/8,5/9,
                 6/7,6/8,6/9,
                 7/9,
                 8/9}
   \thegrid{9}

   \draw (10, 5) node {$\leftrightarrow$};
%    \filldraw (10.5, 8.5) circle (\spw);
    \begin{scope}[shift={(11,6.15)}]
        \thegrid{2}
        \thepoints[(0,1)]{2}
    \end{scope}
    \begin{scope}[shift={(11,2.75)}]
        \thegrid{3}
        \thepoints[(0,2)]{3}
    \end{scope}
    \begin{scope}[shift={(10.5,1)}]
        \thegrid{2}
        \thepoints[(0,1)]{2}
    \end{scope}
    \begin{scope}[shift={(11,-0.3)}]
        \thegrid{1}
        \thepoints[(0,0)]{1}
    \end{scope}
\end{tikzpicture}
\caption{The staircase grid $D_9$, where vertices $(13)$, $(16)$, $(18)$ are
in the independent set. The
vertices which cannot be added to make an independent set are shaded. The induced
subgraphs are shown on the right.}
\label{fig:132staircase_grid_with_shading}
\end{center}
\end{figure}

To obtain a formula for the coefficients of the generating function $F(x,y)$ we
can make the substitution $F = 1 + w$ in the functional equation in
Theorem~\ref{thm:indep_count_132}. Rearranging produces the equation.
\[
	w = x (1+y)\left( \frac{1+w}{1-yw} \right).
\]
This same equation appears in Section 2.2 in Flajolet and Noy~\cite{flanoy} where
they enumerate non-crossing
subgraphs of a complete graph drawn on a polygon. Like there we can apply
Lagrange inversion (see e.g., Comtet~\cite{comtet}) to obtain the coefficients:

\begin{corollary}[Theorem 2 (ii) in Flajolet and Noy~\cite{flanoy}]
  \label{cor:indep_count_132}
  The number of independent sets of size $k$ in the $132$-core of size $n$ is
  given by
  \[
      I(n,k) = \frac{1}{n} \sum_{j=0}^{n-1}
      \binom{n}{k-j} \binom{n}{j+1} \binom{n-1+j}{n-1}.
  \]
\end{corollary}

Setting $y = z/(1-z)$, to allow an arbitrarily long (non-empty) increasing sequence
instead of a vertex in an independent set, and $x = tz$, to track the number of
left-to-right minima, in the generating function in
Theorem~\ref{thm:indep_count_132} gives the
generating function
\begin{equation}\label{eqn:Narayana}
    \frac{z-tz+1-\sqrt{(t-1)^2 z^2 - 2(t+1) +1}}{2z}.
\end{equation}
This is the well-known generating function for the Narayana numbers~\cite[A001263]{oeis}
enumerating the
permutations avoiding $132$ by their number of left-to-right minima. These
numbers are usually arranged in a triangle, whose $n$-th row (starting from $0$)
contains the number of $132$-avoiding permutations of length $n$ with $k$
left-to-right minima (the coefficient of $z^n t^k$). The triangle is often called
a \emph{Catalan-triangle} since the row sums are Catalan numbers, see
Table~\ref{tbl:Narayanatriangle}.

\begin{table}[!hbp]
{\small
\begin{tabular}{lllllllllll}
1 & & & & & & & & & & \\
1 & & & & & & & & & & \\
1 & 1 & & & & & & & & & \\
1 & 3 & 1 & & & & & & & & \\
1 & 6 & 6 & 1 & & & & & & & \\
1 & 10 & 20 & 10 & 1 & & & & & & \\
1 & 15 & 50 & 50 & 15 & 1 & & & & & \\
1 & 21 & 105 & 175 & 105 & 21 & 1 & & & & \\
1 & 28 & 196 & 490 & 490 & 196 & 28 & 1 & & & \\
1 & 36 & 336 & 1176 & 1764 & 1176 & 336 & 36 & 1 & & \\
1 & 45 & 540 & 2520 & 5292 & 5292 & 2520 & 540 & 45 & 1 & \\
1 & 55 & 825 & 4950 & 13860 & 19404 & 13860 & 4950 & 825 & 55 & 1
\end{tabular}
}
\caption{The Catalan triangle of Narayana numbers~\cite[A001263]{oeis}.}
\label{tbl:Narayanatriangle}
\end{table}

Setting $t = 1$ in \eqref{eqn:Narayana}, and thus forgetting the information
about the left-to-right minima, gives the usual generating function of the
Catalan numbers.

Instead of enumerating $\av{132}$ by the number of left-to-right minima we can
enumerate them by the size of the independent set in the corresponding core.
\begin{proposition}
    \label{prop:triangle}
    The number of permutations avoiding $132$ of length $\ell$ produced by inflating an
    independent set of size $k$ is
    \begin{equation*}
        \sum_{n=0}^\ell{}I(n,k)\binom{\ell-n-1}{k-1}
    \end{equation*}
    where $I(n,k)$ is defined in Corollary~\ref{cor:indep_count_132}.
\end{proposition}
\begin{proof}
The sum is taken over the number of left-to-right minima and the binomial
coefficient counts the number of compositions of the remaining $\ell - n$ points
into $k$ non-empty parts, corresponding to the vertices of the independent set.
\end{proof}

The proposition gives a new Catalan triangle, shown in
Table~\ref{tbl:Newtriangle}, which has been added to the Online Encyclopedia of
Integer Sequences~\cite[A262370]{oeis}.
\begin{table}[!htbp]
{\small
\begin{tabular}{llllllllll}
1 & & & & & & & & & \\
1 & & & & & & & & & \\
1 & 1 & & & & & & & & \\
1 & 4 & & & & & & & & \\
1 & 10 & 3 & & & & & & & \\
1 & 20 & 20 & 1 & & & & & & \\
1 & 35 & 77 & 19 & & & & & & \\
1 & 56 & 224 & 139 & 9 & & & & & \\
1 & 84 & 546 & 656 & 141 & 2 & & & & \\
1 & 120 & 1176 & 2375 & 1104 & 86 & & & & \\
1 & 165 & 2310 & 7172 & 5937 & 1181 & 30 & & & \\
1 & 220 & 4224 & 18953 & 24959 & 9594 & 830 & 5 & & \\
1 & 286 & 7293 & 45188 & 87893 & 56358 & 10613 & 380 & & \\
1 & 364 & 12012 & 99242 & 270452 & 264012 & 88472 & 8240 & 105 & \\
1 & 455 & 19019 & 203775 & 747877 & 1044085 & 554395 & 100339 & 4480 & 14
\end{tabular}
}
\caption{The Catalan triangle given by Proposition~\ref{prop:triangle}.}
\label{tbl:Newtriangle}
\end{table}

In the triangle, the right-most numbers in rows $2, 5, 8, 11, 14$ form the
sequence
\[
    1, 1, 2, 5, 14,
\]
and the right-most numbers in rows $1, 4, 7, 10, 13$ form the sequence
\[
    1, 3, 9, 30, 105.
\]
This leads to the following conjecture, where the first 8 terms of the sequences
have been checked.

\begin{conjecture}
    \begin{enumerate}
        \item The right-most numbers in rows numbered $2 + 3i$, $i = 0,1\dots$
        are the Catalan numbers.
        \item The right-most numbers in rows numbered $1 + 3i$, $i = 0,1\dots$
        are the central elements of the $(1,2)$-Pascal
        triangle~\cite[A029651]{oeis}.
    \end{enumerate}
\end{conjecture}

The reader might wonder about the sequence of right-most numbers in the remaining rows.
We have not found a particular structure in that sequence, and
it does not appear in the Online Encyclopaedia of Integer Sequences~\cite{oeis}.

\section{Permutations avoiding $123$}
\label{sec:permutations_avoiding_123}
%!TEX root = patt-gons.tex

Consider drawing a permutation that avoids $123$ on a staircase grid. Take for
example the permutation $639871542$ shown in Figure~\ref{fig:lrmgrid-123}.
\begin{figure}[htp]
\begin{center}
\begin{tikzpicture}[scale=\picscales]
\thepoints[(0,2)]{3}
\labelthepointsbelow[(0,2)]{6,3,1}
\labeltheboxes[(0,3)]{1/2/\permaslabel{0.15}{0.25}{3,2,1}/0/0,
               2/3/\permaslabel{0.15}{0.25}{2,1}/-0.15/0,
               3/3/\permaslabel{0.15}{0.25}{1}/0.15/0}
\thegrid{3}
\end{tikzpicture}

\caption{The permutation $639871542$ drawn on a staircase grid. The
left-to-right minima, $6$, $3$ and $1$, are drawn on a diagonal.}
\label{fig:lrmgrid-123}
\end{center}
\end{figure}
Every rectangular region of boxes is decreasing, instead of increasing in the
case of the $132$-avoiding permutations, and the existence of points in a box
implies boxes southwest and northeast of it must be empty. We are naturally
lead to the following definition:

\begin{definition}\label{def:123core}
    Let $n \geq 0$ be an integer. The \emph{$123$-core of size $n$} is the
  labeled, undirected graph $U_n$ with vertex set $\{ (ij) \colon i = 1, \dots,
  n,\, j = i, \dots, n \}$ and an edge between $(ij)$ and $(k\ell)$ if $i > k$,
  $j < \ell$ and the rectangle $\{k, k+1, \ldots, i\}\times\{j, j+1,
  \ldots,\ell\}$ is a subset of the vertex set.
\end{definition}

\begin{figure}[ht]
\begin{center}
\begin{tikzpicture}[scale=\picscale]
    \thecoreB{1}

    \thecoreB[(2.5,-1)]{2}

    \thecoreB[(6,-2)]{3}

    \thecoreB[(10.5,-3)]{4}

    \thecoreB[(16,-4)]{5}

\end{tikzpicture}

\caption{The $123$-cores of sizes $1, \dots, 5$.}
\label{fig:123-cores}
\end{center}
\end{figure}

See Figure~\ref{fig:cellincore} (right) for the set of neighbours of a cell in a
$132$-core, and Figure~\ref{fig:123-cores} for the $123$-cores of sizes
$n = 1, \dots, 5$.
We use the letter ``$U$'' in anticipation of Definition~\ref{def:cores}.
The $123-$core $U_n$ is isomorphic to $D_n$ (the $132$-core of size $n$)
when $n = 0,1,2,3$. However, $D_4$ is a $5$-cycle, whereas $U_4$ is not.

Even though the $123$-cores are not isomorphic to the $132$-cores in general
they have the same number of cliques of each size. This can be
proven with the same method as Proposition~\ref{prop:cliquecount132}
\begin{proposition}
    \label{prop:cliquecount123}
    The number of cliques of size $k$ in $U_n$ is
    $\binom{n+1}{2k}$.
\end{proposition}

Moreover, it turns out that the number of independent sets of each
size is the same in $123$-cores as in $132$-cores.

\begin{theorem} \label{thm:indep_count_123}
  The number of independent sets of size $k$ in the $123$-core of size $n$ is
  given by the coefficient of $x^ny^k$ in the generating function $F = F(x,y)$
  satisfying the functional equation
  \[
    F = 1 + xF +
    \frac{xyF^2}{1 - y(F - 1)}.
  \]
\end{theorem}

We leave it to the reader to prove this with the method used to prove
Theorem~\ref{thm:indep_count_132}, with the modification that we consider the
number of non-empty boxes on the diagonal instead of the top row, see
Figure~\ref{fig:123staircase_grid_with_shading}.

\begin{figure}[h]
\begin{center}
\begin{tikzpicture}[scale=\picscales]
   \thepoints[(0,8)]{9}
   \labeltheboxes[(0,9)]{3/3/\permaslabel{0.2}{0.25}{1}/0/0,
                  6/6/\permaslabel{0.2}{0.25}{1}/0/0,
                  8/8/\permaslabel{0.2}{0.25}{1}/0/0}
  \shadetheboxes[(0,9)]{1/1,1/4,1/5,1/6,1/7,1/8,1/9,
                 2/2,2/4,2/5,2/6,2/7,2/8,2/9,
                 3/7,3/8,3/9,
                 4/4,4/7,4/8,4/9,
                 5/5,5/7,5/8,5/9,
                 6/9,
                 7/7,7/9,
                 9/9}
   \thegrid{9}

   \draw (10, 5) node {$\leftrightarrow$};
%    \filldraw (10.5, 8.5) circle (\spw);
    \begin{scope}[shift={(11,7.15)}]
        \thegrid{2}
        \thepoints[(0,1)]{2}
    \end{scope}
    \begin{scope}[shift={(11,3.75)}]
        \thegrid{3}
        \thepoints[(0,2)]{3}
    \end{scope}
    \begin{scope}[shift={(10.5,2)}]
        \thegrid{2}
        \thepoints[(0,1)]{2}
    \end{scope}
    \begin{scope}[shift={(11,0.7)}]
        \thegrid{1}
        \thepoints[(0,0)]{1}
    \end{scope}
\end{tikzpicture}
\caption{Illustration of how to prove Theorem~\ref{thm:indep_count_123} using
a modification of the method in the proof of Theorem~\ref{thm:indep_count_132}.}
\label{fig:123staircase_grid_with_shading}
\end{center}
\end{figure}

Theorem~\ref{thm:indep_count_123} implies that the bivariate generating function in
Theorem~\eqref{thm:indep_count_132}, the enumerations in
Proposition~\ref{prop:triangle} and Corollary~\ref{cor:indep_count_132},
and the Catalan triangles in Tables~\ref{tbl:Narayanatriangle} and~\ref{tbl:Newtriangle}
also apply to permutations avoiding $123$.

\section{Permutations avoiding $1324$}
\label{sec:permutations_avoiding_1324}
%!TEX root = patt-gons.tex

In this section we generalize the cores defined above with the goal of
enumerating subclasses of $\av{1324}$. We start with the theory of these new
cores. The first subclass considered is $\av{1324, 2143}$ in
subsection~\ref{sub:Smooth permutations}. In the following
two subsections we enumerate $\av{1234, 1324, 2143}$ and $\av{1234, 1324, 1432, 3214}$.
The method of proof for all of these classes is to consider the subsequence of points
in a permutation containing the left-to-right minima and the right-to-left maxima.
We then find a structural description of these \emph{boundaries}, and use the
generalized cores to find the enumeration of the original class.
The reason to focus on subclasses of the permutation class $\av{1324}$, is because it is the
only unenumerated class avoiding a single length 4 pattern. The theory we develop
also applies to subclasses of $\av{1234}$, which has a known enumeration,
first found by Gessel~\cite{Gessel:1990bq},
and subclasses of $\av{1234, 1324}$, which is conjectured to be
enumerated by a non-$D$-finite generating function, see Albert et al.~\cite[Section 6.2]{resconts}.

For these classes of permutations, the staircase encoding defined above is no
longer unique, the simplest example being the permutations $123$ and $132$ which
belong to all three classes and have the same staircase encoding. To remedy this
we consider the right-to-left maxima (rlm) of a permutation:

\begin{definition}
	Given a permutation $\pi$, we define its \emph{boundary},
	$\bd{\pi}$, as the standardization of the subsequence of $\pi$ containing the
	left-to-right minima and the right-to-left maxima.
\end{definition}

For example if $\pi =
(11)(12)(13)(10)293671845$, as in Figure~\ref{fig:niceboundary}, then $\bd{\pi}$ is
the standardization of $(11)(13)(10)29185$, i.e., $87625143$. By
construction $\bd{\pi}$ avoids $123$, and every permutation that avoids $123$
is its own boundary.

\begin{definition}
	Given a permutation $\pi$ that avoids $123$ we define its \emph{boundary grid},
	$\bg{\pi}$, as: (1) the collection of $1$-by-$1$ boxes whose corners have integer
	coordinates, with the requirement that the lower left corner of each box is
	northeast of a left-to-right minimum, and the upper right corner of each box is
	southwest of a right-to-left maximum; and (2) the collection of points which
	are either a left-to-right minimum, or a right-to-left maximum.
\end{definition}

For example, the boundary grid of $87625143$ is the boundary grid in
Figure~\ref{fig:niceboundary}.
For a permutation $\pi$ we create its (\emph{boundary}) \emph{encoding}, $\benc{\pi}$,
by writing the number of points in each box of the boundary grid of $\bd{\pi}$; see Figure~\ref{fig:niceboundary} (right).

\begin{figure}[ht]
\begin{center}
\begin{tikzpicture}[scale=\picscales]

	\begin{scope}[shift={(1,3)}]
		\thegrid{1}
		\thepoints[(0,0)]{1}
		\thepoints[(1,1)]{1}
		\thepoints[(2,-1)]{1}
		\labelthepointsbelow[(0,-0.1)]{11}
		\labelthepointsabove[(1,1.1)]{13}
		\labelthepointsbelow[(2,-1.1)]{10}
		\labeltheboxes[(0,1)]{1/1/\permaslabel{0.2}{0.25}{1}/0/0}
	\end{scope}

    \begin{scope}[rotate=-90, yscale=-1, shift={(0,-8)}]
	    \labelthepointsbelow[(0,1)]{8,5}
		\labelthepointsbelow[(-1,3)]{9}
        \labelthepointsabove[(2,4)]{2}
        \labelthepointsabove[(3,2)]{1}
        \thepoints[(1,0)]{1}
        \thepoints[(0,1)]{1}
        \thepoints[(-1,3)]{1}
        \thepoints[(2,4)]{1}
        \thepoints[(3,2)]{1}
        \thegridext{2}
        \labeltheboxes[(0,4)]{2/2/\permaslabel{0.2}{0.25}{1}/0.15/0.15,
                    4/2/\permaslabel{0.2}{0.25}{1}/-0.15/-0.15,
                    2/1/\permaslabel{0.2}{0.25}{1,2}/0/-0.15}
    \end{scope}

    \begin{scope}[shift={(11,3)}]
		\thegrid{1}
		\thepoints[(0,0)]{1}
		\thepoints[(1,1)]{1}
		\thepoints[(2,-1)]{1}
		\labeltheboxes[(0,1)]{1/1/1/0/0}
	\end{scope}

    \begin{scope}[rotate=-90, yscale=-1, shift={(-1,-6)}]
        \thegridext{2}
    \end{scope}

    \begin{scope}[shift={(10,0)}]
    \begin{scope}[rotate=-90, yscale=-1, shift={(0,-8)}]
        \thepoints[(1,0)]{1}
        \thepoints[(0,1)]{1}
        \thepoints[(-1,3)]{1}
        \thepoints[(2,4)]{1}
        \thepoints[(3,2)]{1}
        \thegridext{2}
        \labeltheboxes[(0,4)]{2/2/1/0/0,
                    4/2/1/0/0,
                    2/1/2/0/0}
    \end{scope}
    \begin{scope}[rotate=-90, yscale=-1, shift={(-1,-6)}]
        \thegridext{2}
    \end{scope}
    \end{scope}
\end{tikzpicture}

\caption{On the left, the permutation $(11)(12)(13)(10)293671845$ is drawn on a boundary grid
highlighting its lrm's and rlm's. On the right, we have the boundary encoding
of the same permutation.}
\label{fig:niceboundary}
\end{center}
\end{figure}

We generalize the construction of the $132$- and $123$-cores above:

\begin{definition} \label{def:cores}
  A \emph{grid} is a collection of $1$-by-$1$ boxes whose corners
  have integer coordinates, along with a collection of points with integer coordinates,
  such that no two points lie on the same horizontal or vertical line.\footnote{For consistency
  with the staircase grids we use matrix notation here as well. We use the convention that
  the left-most box, or point, of a grid has $x$ coordinate $1$, and the top-most box, or
  point, of a grid has $y$ coordinate $1$.}
  Given such a grid $B$ we define the
  \begin{enumerate}
    \item \emph{down-core} as the graph $D(B)$ whose vertices are the boxes in the grid
    and an edge between boxes $(ij)$, $(k\ell)$ if $i < k$, $j < \ell$ and
    the rectangle $\{i, i+1, \dots, k\} \times \{j, j+1, \dots, \ell\}$ is
    a subset of $B$.
    \item \emph{up-core} as the graph $U(B)$ whose vertices are the boxes in the grid
	and an edge between boxes $(ij)$, $(k\ell)$ if $i > k$, $j < \ell$
    and the rectangle $\{k, k+1, \dots, i\} \times \{j, j+1, \dots, \ell\}$ is
    a subset of $B$.
    \item \emph{updown-core} as the graph $UD(B)$ whose vertices are the boxes in the grid
    and an edge between boxes $(ij)$, $(k\ell)$ if one of the
    following conditions is satisfied:
    (a) There is an edge between $(ij)$, $(k\ell)$ in $D(B)$;
    (b) there is an edge between $(ij)$, $(k\ell)$ in $U(B)$;
    (c) if $i = k$, $\ell < j$ and rectangle $\{i \} \times \{\ell, \ell+1, \dots, j\}$
    is a subset of $B$;
    (d) if $i < k$, $j = \ell$ and rectangle $\{i, i+1, \dots, k \} \times \{ j \}$
    is a subset of $B$.
  \end{enumerate}
\end{definition}

See Figure~\ref{fig:allthecores} for an example of the different types of cores.

\begin{figure}[h]
	\centering
	\begin{tikzpicture}[scale=\picscales]

			\someboxes{1/0, 3/0, 4/0, 5/0, 6/0,
					   1/1, 2/1, 3/1, 4/1, 5/1, 6/1, 7/1, 8/1,
					   1/2, 2/2, 3/2, 4/2, 5/2, 6/2, 7/2,
					   2/3, 3/3, 4/3, 5/3, 6/3, 7/3, 8/3,
					   1/4, 3/4, 4/4, 5/4, 6/4,
					   6/5, 8/5,
					   8/6}
			\somepoints{1/0, 2/3, 9/4, 7/6, 3/7}
        	\node[mynode] (4 5) at (5-0.5,3-0.5) {\tiny$5 4$};
        	\shadetheboxes[(0,7)]{3/4,
        	                      4/3, 4/4,
        	                      6/6, 6/7, 6/8,
        	                      7/6, 7/7}

			\begin{scope}[shift={(10.3,0)}]
			\someboxes{1/0, 3/0, 4/0, 5/0, 6/0,
					   1/1, 2/1, 3/1, 4/1, 5/1, 6/1, 7/1, 8/1,
					   1/2, 2/2, 3/2, 4/2, 5/2, 6/2, 7/2,
					   2/3, 3/3, 4/3, 5/3, 6/3, 7/3, 8/3,
					   1/4, 3/4, 4/4, 5/4, 6/4,
					   6/5, 8/5,
					   8/6}
			\somepoints{1/0, 2/3, 9/4, 7/6, 3/7}
        	\node[mynode] (4 5) at (5-0.5,3-0.5) {\tiny$5 4$};
        	\shadetheboxes[(0,7)]{3/6, 3/7,
        	                      4/6, 4/7, 4/8,
        	                      6/2, 6/3, 6/4,
        	                      7/4}
			\end{scope}

			\begin{scope}[shift={(20.6,0)}]
			\someboxes{1/0, 3/0, 4/0, 5/0, 6/0,
					   1/1, 2/1, 3/1, 4/1, 5/1, 6/1, 7/1, 8/1,
					   1/2, 2/2, 3/2, 4/2, 5/2, 6/2, 7/2,
					   2/3, 3/3, 4/3, 5/3, 6/3, 7/3, 8/3,
					   1/4, 3/4, 4/4, 5/4, 6/4,
					   6/5, 8/5,
					   8/6}
			\somepoints{1/0, 2/3, 9/4, 7/6, 3/7}
        	\node[mynode] (4 5) at (5-0.5,3-0.5) {\tiny$5 4$};
        	\shadetheboxes[(0,7)]{3/6, 3/7,
        	                      4/6, 4/7, 4/8,
        	                      6/2, 6/3, 6/4,
        	                      7/4}
        	\shadetheboxes[(0,7)]{3/4,
        	                      4/3, 4/4,
        	                      6/6, 6/7, 6/8,
        	                      7/6, 7/7}
        	\shadetheboxes[(0,7)]{3/5,
        	                      4/5,
        	                      5/2, 5/3, 5/4, 5/6, 5/7, 5/8,
        	                      6/5,
        	                      7/5}
			\end{scope}

	\end{tikzpicture}
	\caption{The cell $(54)$ with its set of neighbours shaded; on the
	left in the down-core, in the middle in the up-core, and on the right in the updown-core.}
	\label{fig:allthecores}
\end{figure}

We note that the $132$-core is the down-core for the staircase grid while the
$123$-core is the up-core for the same grid.

For two permutations $\sigma$ and $\tau$ the \emph{skew-sum}, $\sigma \ominus \tau$,
is the permutation created by concatenating $\sigma$ and $\tau$
after adding the length of $\tau$ to each element of $\sigma$. For example, $1432
\ominus 3142 = 58763142$. A permutation is \emph{skew-decomposable} if it can be
written as a non-trivial skew-sum. By convention, the empty permutation is
skew-decomposable.

\begin{lemma}
	A $123$-avoiding permutation $\pi$ of length greater than $1$ is skew-indecomposable
	if and only if
	$UD(\bg{\pi})$ is connected and every left-to-right minimum is attached to some
	$1$-by-$1$ box in $\bg{\pi}$.
\end{lemma}

\begin{proof}
	Assume $\pi$ is skew-decomposable as $\pi = 	\sigma \ominus \tau$ where $\sigma$ has
	length $s$, and $\tau$ is skew-indecomposable. If $\tau = 1$ then it is a left-to-right
	minimum not attached to a box. Otherwise there is at least one box in the boundary grid
	of $\tau$. Either $\sigma$ has a boundary grid consisting of just integer points,
	in which case we are done; or it has at least one box which is disconnected from the box
	in $\tau$.

	For the other direction, if there is a left-to-right minimum that is not
	attached to a box then it is a component in the skew-decomposition of $\pi$.
	Otherwise there must be exist two distinct subsets, $B_1$, and $B_2$, of boxes which are
	connected components. We can assume that all the boxes in $B_1$ are to the left and above
	all the boxes in $B_2$. The left-to-right minima, and the right-to-left maxima, attached
	to boxes in $B_1$ must be to the left and above the corresponding points attached to
	boxes in $B_2$. This implies that the points attached to $B_1$ are in a different component
	than the points attached to $B_2$ are, in the skew-decomposition of $\pi$.
\end{proof}

For the subclasses considered below we will begin with finding the generating function
for the skew-indecomposable permutations in the class. We will need rely on the following
easily proven remark to translate such a generating function to the entire class.

\begin{remark}\label{rmk:skewclosure}
  If a permutation class $\mathcal{C}$ has basis $P$ and the patterns in $P$ are
  skew-indecomposable then the skew-sum of two permutations from $\mathcal{C}$ is another
  permutation in $\mathcal{C}$. For such classes the generating functions enumerating the
  class, $C$, and the skew-indecomposable permutations in the class, $C_{ind}$ are related
  by the formula
  \[
  	C = \frac{1}{1-C_{ind}},
  \]
  because every permutation in the class has a unique decomposition into skew-indecomposable
  permutations. This property holds for the classes studied in this section.
\end{remark}

\begin{definition}
	Let $B$ be the boundary grid of a permutation $\pi$ avoiding $123$.
	\begin{enumerate}
		\item Let $\dperms{B}$ be the set of permutations obtained by choosing a
        weighted independent set in the down-core $D(B)$ and inflating the vertices
        of the independent set into a increasing sequence of points, whose length
        is determined by the weight of the vertex, while making sure every rectangular
        region inside the grid is increasing.
        \item Likewise, we define $\uperms{B}$ using the up-core, by inflating weights
        into decreasing sequences.
		\item Finally we define $\udperms{B}$ using the updown-core, by inflating an (unweighted)
		independent set into a single point for each vertex.
	\end{enumerate}
\end{definition}

For example, consider the boundary grid on the right in Figure~\ref{fig:niceboundary},
with the set of vertices $(11), (55), (65), (67)$ with weights $1, 2, 1, 1$. These
vertices form an independent set in the down-core of this boundary grid. If we inflate
the vertices into increasing sequences with length determined by the weights (making
sure every rectangular region inside the grid contains an increasing sequence of points)
we reclaim the permutation $(11)(12)(13)(10)293671845$. This same set of vertices is an
independent set in the up-core of the same grid. If we instead inflate the vertices into
decreasing sequences we obtain the permutation $(11)(12)(13)(10)297641835$. Finally,
these vertices are not an independent set in the updown-core on the grid, since there
are edges between the vertices $(55)$, $(65)$, as well as $(65)$, $(67)$.

With these definitions we can state the main result on these new cores:

\begin{proposition}\label{prop:disjoint}
	\begin{enumerate}
	\item There is a bijection from the set of permutations avoiding $1324$
	and pairs $(B, I)$ where $B$ is the boundary grid of a $123$-avoiding permutation
	and $I$ is a weighted independent set in the down-core of $B$. In other words
	\[
    	\av{1324} = \bigsqcup_{\pi \in \av{123}} \dperms{\bg{\pi}}.
	\]\label{prop:disjoint1}
	\item Similarly, there is a bijection from the set of permutations avoiding $1234$
	and pairs $(B, I)$ where $B$ is the boundary grid of a $123$-avoiding permutation
	and $I$ is a weighted independent set in the up-core of $B$. In other words
	\[
    	\av{1234} = \bigsqcup_{\pi \in \av{123}} \uperms{\bg{\pi}}.
	\]\label{prop:disjoint2}
	\item Finally, there is a bijection from the set of permutations avoiding $1234$ and $1324$
	and pairs $(B, I)$ where $B$ is the boundary grid of a $123$-avoiding permutation
	and $I$ is an (unweighted) independent set in the updown-core of $B$. In other words
	\[
    	\av{1234, 1324} = \bigsqcup_{\pi \in \av{123}} \udperms{\bg{\pi}}.
	\]\label{prop:disjoint3}
	\end{enumerate}
\end{proposition}

\begin{proof}
	We only prove the statement for $\av{1324}$ as the other statements can be proved by
	the same arguments. Let $\pi$ be a permutation in $\av{1324}$. The boundary $\bd{\pi}$
	is a $123$-avoiding permutation	and we need to prove that the boundary encoding, $\benc{\pi}$,
	of $\pi$ corresponds to a weighted independent set in the down-core of $\bg{\pi}$. It is
	clear that the encoding corresponds to a weighted set of vertices in the core. Assume that
	there is an edge between two of these vertices, say $(ij)$, $(k\ell)$, with $i < k$,
	$j < \ell$. By the definition of the down-core this implies that the rectangle
	$\{i, i+1, \dots, k\} \times \{j, j+1, \dots, \ell\}$ is a subset of the core. By
	the definition of the boundary grid there is at least one left-to-right minimum southwest
	of this rectangle, as well as a right-to-left maximum to the northeast of it. But these
	two boundary points, together with a point from the box $(ij)$ and a point from the box
	$(k\ell)$ form an occurrence of $1324$. This is a contradiction, and therefore there cannot
	be an edge between any of these vertices. This shows that $\av{1324}$ is a subset of the
	disjoint union on the right. The other subset relation is similar and left to the reader.
\end{proof}

Our main focus will be on subclasses of the first and third permutation classes
in the equations above, as their enumerations are unknown. We start by
considering the subclass $\av{1324, 2143}$ of $\av{1324}$.

\subsection{Down-cores and smooth permutations}
\label{sub:Smooth permutations}
The \emph{smooth} permutations are those that correspond to smooth Schubert varieties.
Sandhya and Lakshmibai~\cite{lakshmibai} showed that these permutations
are the class $\av{1324, 2143}$. The enumeration first appeared in an
unpublished preprint of Haiman~\cite{SmoothHaiman} and is given by the generating
function
\begin{equation} \label{eqn:smoothgenfunc}
    \frac{1-5z+3z^2+z^2\sqrt{1-4z}}{1-6z+8z^2-4z^3}.
\end{equation}
Bousquet-M\'{e}lou and Butler~\cite{SmoothBousquet} provided an independent
proof of this generating function, and Slofstra and
Richmond~\cite{SmoothRichmond} have found the enumeration of smooth Schubert
varieties of all classical finite types, the case of permutations being type
$A$. Below we will rederive the generating function~\eqref{eqn:smoothgenfunc}
while keeping track of the number of boundary points and size of the
independent set.

Our approach uses the equation
\begin{equation} \label{eqn:av13242143}
    \av{1324,2143} = \bigsqcup_{\pi \in \av{123,2143}} \dperms{\bg{\pi}},
\end{equation}
which follows from Proposition~\ref{prop:disjoint}~\eqref{prop:disjoint1}, and the fact that
a $1324$-avoiding permutation $\pi$ contains $2143$ if and only if $\bd{\pi}$ contains $2143$.
West~\cite{West:1996cy} proved that the enumeration of the boundaries $\av{123,2143}$ is
given by the alternate Fibonacci numbers~\cite[A001519]{oeis}. The following lemma
gives a structural description of the boundaries suitable for our purposes.

\begin{definition}\label{lem:symm}
    For a boundary grid $B$, we define the \emph{reflected} grid, $B^\prime$, as
    the grid obtained by reflecting $B$ along a southeast-northwest diagonal.
\end{definition}
It follows from Definition~\ref{def:cores} that $D(B)$ is isomorphic to $D(B^\prime)$,
$U(B)$ is isomorphic to $U(B^\prime)$, and $UD(B)$ is isomorphic to $UD(B^\prime)$.

\begin{lemma} \label{lem:skew-ind1232143}
    Let $\pi$ be a skew-indecomposable permutation in $\av{123,2143}$ of
    length greater than one. The boundary grid of $\pi$ is a union of a sequence
    of grids, which alternate between staircase grids and reflected staircase grids,
    which share their extremal boxes; see
    Figure~\ref{fig:smooth-block} (left).
\end{lemma}

\begin{proof}
	Let $\pi$ be a skew-indecomposable permutation that avoids $123$ and $2143$. We assume
	$\pi$ starts with a descent, i.e., $\pi_1 > \pi_2$. Let
	$\pi_1 > \pi_2 > \dots > \pi_s$ be the initial decreasing sequence of $\pi$,
	so $\pi_{s+1} > \pi_s$ (if $\pi_s$ were the last point in $\pi$ then it
	would be skew-decomposable). If $\pi_1 > \pi_{s+1}$ then there must be be a point $p$
	in $\pi$ to the right of $\pi_{s+1}$ with $p > \pi_1$, otherwise $\pi$ is skew-decomposable.
	Then $\pi_s \pi_{s+1} p$ is an occurrence of $123$, a contradiction.
	Therefore $\pi_{s+1} > \pi_1$. If $\pi_{s+1}$ is the last point in the permutation
	the claim holds. If $\pi_{s+1}$ is not the last point in the permutation then
	the difference in values between $\pi_i$ and $\pi_{i+1}$ for $i = 1, \dots, s-2$
	must be $1$, for if this failed for a specific $i$ the subsequence
	$\pi_{i+1} \pi_{i+2} \pi_{s+1} p$ (where $p$ is a point in the permutation with
	value between $\pi_i$ and $\pi_{i+1}$ and index greater than $s+1$) is an occurrence
	of $2143$. It follows
	that rows $1, 2, \dots, s$ in $\bg{\pi}$ have the shape of a staircase grid $B_s$.
	If $\pi_{s-1} - \pi_s = 1$ then this is the entirety of the boundary grid of $\pi$.
	Otherwise we consider the decreasing sequence of points northeast of the point
	$\pi_s$. A very similar argument shows that rows $s, s+1, \dots, s+r$ (where $r+1$
	is the length of the decreasing sequence), up to and including column $s+r$ have
	the shape of a reflected staircase grid $B_r'$. This alternation between staircase
	grids, and their reflections, persists until the end of the permutation.

	The case where $\pi$ starts with an ascent, i.e., $\pi_1 < \pi_2$ can be done with
	the same argument, the only difference is that the top-most box, which is then alone
	in the top row, is then the first staircase grid $B_1$.
\end{proof}

To enumerate the class $\av{1324,2143}$ we need to enumerate the independent sets
of the down-cores on the boundary grids described by the previous lemma. To this end
we need to split those grids into smaller grids of the following type:
Let $EB_n$ be the grid obtained by doubling the top-most column in
the staircase grid $B_n$, see Figure~\ref{fig:eb}. Let $ED_n$ be the down-core of this grid.\footnote{``E'' is
for extended.}

\begin{figure}[ht]
\begin{center}
\begin{tikzpicture}[scale=\picscales]
    \thepoints[(0,3)]{4}
    \thegridextf{4}
\end{tikzpicture}
\caption{The grid $EB_4$.}
\label{fig:eb}
\end{center}
\end{figure}

\begin{lemma} \label{lem:G}
    Let $G = G(x,y)$ be the generating function where the coefficient of
    $x^n y^k$ is the number of independent sets of size $k$ in $ED_n$. Then
    \[
        G = \frac{F-1}{(1+y)x},
    \]
    where $F = F(x,y)$ is the generating function in
    Theorem~\ref{thm:indep_count_132}.
\end{lemma}

\begin{proof}
    If a box (with an attached left-to-right minimum) is added on the left
    of the top-most row of $EB_n$, then no new edges are added to the core.
    In fact we obtain the core $B_{n+1}$. Therefore
    $F = (1+y)xG + 1$ and solving for $G$ gives the claimed equation.
\end{proof}

With this in hand we can enumerate the independent sets of the cores in
Lemma~\ref{lem:skew-ind1232143}.

\begin{lemma}\label{lem:pind}
    Let $P_{ind} = P_{ind}(x,y)$ be the generating function where the
    coefficient of $x^n y^k$ is the number of independent sets of size $k$ in
    down-cores on boundaries of skew-indecomposable permutations from
    $\Av_n(123,2143)$. Then
    \[
        P_{ind} = \frac{x(F-1)}{2-G} + x,
    \]
    where $F = F(x,y)$ is the generating function in
    Theorem~\ref{thm:indep_count_132}, and $G = G(x,y)$ is the generating
    function in Lemma~\ref{lem:G}.
\end{lemma}

\begin{figure}[ht]
\begin{center}
\begin{tikzpicture}[scale=\picscales]
    \thepoints[(0,1)]{2}
    \thepoints[(3,2)]{1}
    \thegridext{2}
    \begin{scope}[rotate=-90, yscale=-1, shift={(0,-5)}]
        \thepoints[(1,1)]{2}
        \thepoints[(4,3)]{1}
        \thegridext{3}
    \end{scope}
    \begin{scope}[shift={(5,-5)}]
        \thepoints[(1,0)]{1}
        \thepoints[(2,2)]{1}
        \thegrid{2}
    \end{scope}

    \begin{scope}[shift={(11,-1)}]
        \begin{scope}[shift={(0,0.2)}]
            \thepoints[(0,2)]{2}
            \thepoints[(3,3)]{1}
            \thegrid{3}
        \end{scope}
        \begin{scope}[rotate=-90, yscale=-1, shift={(0,-6)}]
            \thepoints[(0,2)]{2}
            \thepoints[(3,4)]{1}
            \thegridextf{3}
        \end{scope}
        \begin{scope}[shift={(6.2,-4)}]
            \thepoints[(0,0)]{1}
            \thepoints[(1,2)]{1}
            \thegridextf{1}
        \end{scope}
    \end{scope}
\end{tikzpicture}
\caption{A skew-indecomposable boundary in $\av{123,2143}$ is shown on the left, and on
the right how it is decomposed in the proof of Lemma~\ref{lem:pind}.}
\label{fig:smooth-block}
\end{center}
\end{figure}

\begin{proof}
	As in Lemma~\ref{lem:skew-ind1232143} we consider a skew-indecomposable permutation $\pi$
	that avoids $123$ and $2143$.
	Consider the top-most staircase grid $B_n$ described in the proof of
	Lemma~\ref{lem:skew-ind1232143} (see the grid on the left in Figure~\ref{fig:smooth-block}).
	We separate this grid from the boundary grid of $\pi$. The top-most point
	of $\pi$ makes up for the missing left-to-right minimum that should be attached to the
	bottom-most box of $B_n$, see Figure~\ref{fig:smooth-block} (right).
	The next piece of the boundary grid is a reflected $EB_m$ for some $m \geq 0$, and again
	the total number of points attached to the grid is what it should be, even though
	they do not technically appear in the correct place. We keep going and next have another
	extended boundary grid, and then this alternates between extended boundary grids and their
	reflections. We proceed in this manner until the boundary grid of $\pi$ is exhausted. Note
	that there will be one extra boundary point left over.

	Since the core of a reflected grid is isomorphic to the core of the unreflected
	grid we get that the enumeration of independent sets is the same in both. Thus
	the enumeration of the core on the boundary grid of $\pi$ is
	\[
		P_{ind} = (F-1) \cdot \frac{1}{1-(G-1)} \cdot x + x,
	\]
	where $(F-1)$ enumerates a non-empty staircase grid, $1/(1-(G-1))$ accounts for the
	(possibly empty) sequence of extended staircase grids (where every other grid is reflected),
	and then we multiply by $x$ to track the extra boundary point. We add
	$x$ for the empty grid, representing the permutation $1$.
\end{proof}

\begin{proposition}\label{prop:smooth}
    Let $P = P(x,y)$ be the generating function where the coefficient of
    $x^n y^k$ is the number of independent sets of size $k$ in down-cores on
    boundaries given by a permutation in $\Av_n(123,2143)$. Then
    \[
        P = \frac{1}{1-P_{ind}}.
    \]
    By setting $x = z$ and $y = z/(1-z)$ we recover the generating function
    \eqref{eqn:smoothgenfunc} for the number of $(1324,2143)$-avoiding
    permutations.
\end{proposition}
\begin{proof}
Since the patterns $1324$ and $2143$ are skew-indecomposable this follows from
Remark~\ref{rmk:skewclosure}.
\end{proof}

\subsection{Updown-cores and the class $\av{1234,1324,2143}$}
\label{sub:The subclass av{1234,1324,2143}}
The permutation class $\av{1234,1324}$ is conjectured to have a non-$D$-finite
generating function by Albert et. al~\cite[Section 6.4]{resconts}. In this
section, we show that the subclass that avoids $2143$ has the rational
generating function
\begin{equation} \label{eqn:nicegenfunc}
    \frac{1 - 3z - 2z^3}{1 - 4z + 2z^2 - 2z^3 + z^4}.
\end{equation}
Similar to the previous section we use the equation
\[
    \av{1234, 1324, 2143} = \bigsqcup_{\pi \in \av{123, 2143}} \udperms{\bg{\pi}},
\]
which follows from Proposition~\ref{prop:disjoint}~\eqref{prop:disjoint3}.
Lemma~\ref{lem:skew-ind1232143} still describes the boundaries of these
permutations, and we will need to perform the same type of analysis as in the previous
section, with the difference that we need to understand the updown-cores of
boundary grids of permutations in $\av{123, 2143}$.

\begin{lemma}\label{lem:udstair}
	Let $R = R(x,y)$ be the generating function where the coefficient of $x^n y^k$
	is the number of independent sets of size $k$ in updown-cores of $B_n$.
	Then
	\[
		R = 1 + xR + \frac{xyR}{1-x}.
	\]
\end{lemma}

\begin{proof}
	The empty grid $B_0$ has one independent set, namely the empty set which gives
	the first term $x^0 y^0 = 1$ in the equation above. Assume $n > 0$.
	An independent set in $UD(B_n)$ can contain at most one vertex
	from the top row. If it contains no vertex then the remaining vertices form
	a graph isomorphic to $UD(B_{n-1})$. This case gives the term $xR$.
	Otherwise, there is exactly one vertex in the
    top row, see Figure~\ref{fig:proofR}. If the vertex is in box $(1j)$ then it is connected
    to every vertex in
    row $1, 2, \dots, j$, and no other vertices. The induced subgraph on the vertices
    in the remaining rows is a graph isomorphic to $UD(B_{n-j})$. There are $j$
		left-to-right minima to the left of the box $(1j)$. This case
    gives the last term in the equation: $x/(1-x)$ for the minima to the left of $(1j)$,
    $y$ for the box $(1j)$, and $R$ for the subgraph below.
\end{proof}

\begin{figure}[h]
\begin{center}
\begin{tikzpicture}[scale=\picscales]
   \thepoints[(0,4)]{5}
   \labeltheboxes[(0,5)]{1/3/\permaslabel{0.2}{0.25}{1}/0/0}
  \shadetheboxes[(0,5)]{1/1,1/2,1/4,1/5,
                 2/2,2/3,2/4,2/5,
                 3/3,3/4,3/5}
   \thegrid{5}

\end{tikzpicture}
\caption{Illustration of the proof of Lemma~\ref{lem:udstair}.}
\label{fig:proofR}
\end{center}
\end{figure}

\begin{lemma}\label{lem:qind}
    Let $Q_{ind} = Q_{ind}(x,y)$ be the generating function where the
    coefficient of $x^n y^k$ is the number of independent sets of size $k$ in
    updown-cores on boundaries of skew-indecomposable permutations in
    $\Av_n(123,2143)$. Then
    \[
        Q_{ind} = x + x^2 + x^2y + Q_{up} + Q_{down},
    \]
    where $Q_{up} = Q_{up}(x,y)$, $Q_{down} = Q_{down}(x,y)$ satisfy the
    equations
    \[
        Q_{up} = xR - x - x^2 - x^2y + \frac{(xR - x)Q_{down}}{x} + \frac{x^2y R (Q_{up} + Q_{down} + x^2 + x^2y)}{1-x},
    \]
    \[
        Q_{down} = xR - x - x^2 - x^2y + \frac{(xR - x)Q_{up}}{x} + \frac{x^2y R (Q_{up} + Q_{down} + x^2 + x^2y)}{1-x},
    \]
    and $R = R(x,y)$ is given in Lemma~\ref{lem:udstair}.
\end{lemma}

\begin{proof}
	As before we will use $x$ to track the number of boundary points and $y$ to track the
	size of independent sets.
	Recall the description given in Lemma~\ref{lem:skew-ind1232143} of the boundary grid $B$ of a
	skew-indecomposable permutation $\pi$ in $\Av(123,2143)$. There are three types of grids
	we distinguish here:
	\begin{enumerate}
		\item First, when the length of $\pi$ is one or two. Then we can
	directly see that we get the terms $x + x^2(1+y)$ in the equation for $Q_{ind}$.
		\item Next we assume that $\pi$ starts with
	a descent so we can separate a staircase grid $B_b$ (with $b>1)$ from the entire grid.
	Let $Q_{up}$ to be the generating function enumerating the number of independent sets
	for this type.
		\item Finally, $\pi$ could start with a ascent. Then we may consider the leftmost part of
	the boundary grid $B$ as a reflected staircase grid $B_b'$ (with $b>1$). Let $Q_{down}$ be
	the generating function enumerating the number of independent sets for this type.
	\end{enumerate}

	We derive the functional equation satisfied by $Q_{up}$.
	If the boundary grid $B$ has only one part in its skew-sum, then it is isomorphic to
	$B_b$. This gives the terms $xR - x - x^2 - x^2y$ in the equation for $Q_{up}$.
	Let $B$ be a boundary with at least two parts in its skew-decomposition, where the
	left-most part is some $B_b$ with $b > 1$. Let $S$ be the set
    of vertices $S = \{(ib) | i = 1, \ldots, b - 1 \}$, in the last column of $B_b$.
	These vertices are in the same column, so an independent set in $UD(B)$ contains at most one
	vertex from $S$. We consider two cases.
	\begin{enumerate}
		\item[C1]\label{lem:Qind1} Assume that the independent set contains no vertex in $S$.
		If we partition the vertices in $UD(B)$ as the sets $L = \{(ij) | i \leq b - 1 \text{ and } (ij) \in UD(B) \}$
		and $R = \{(ij) | i \geq b, j \geq b  \text{ and } (ij) \in UD(B)\}$,
		then there will be no edges connecting between $L$ and $R$. Moreover, the induced subgraph of $L$ is isomorphic to $UD(B_{b-1})$ and the induced subgraph of $R$ is isomorphic to $UD(B')$ for some boundary grid $B'$ in $Q_{down}$. This gives the $\frac{(xR - x)Q_{down}}{x}$ term in the functional equation. An example of this case
		is shown on the left in Figure~\ref{fig:Qproof}.
		\item[C2]\label{lem:Qind2} Assume that the independent set contains a vertex from $S$, say $bc$, then it can contain no vertices in the set $M = \{(ij) | c \leq j \leq b \text{ and } (ij) \in UD(B) \}$. Partition the remaining vertices into the sets $L =\{(ij) | j \leq c \text{ and } (ij) \in UD(B) \}$ and $R = \{(ij) | j > b \text{ and } (ij) \in UD(B) \}$. There are no edges between $L$ and $R$, and moreover, the induced subgraph of $L$ is isomorphic to $UD(B_{c-1})$, and the induced subgraph of $R$ is isomorphic to $UD(B)$ for some boundary grid $B$ of a skew-indecomposable permutation in $\Av(123, 2143)$. This implies the $\frac{x^2y R (Q_{up} + Q_{down} + x^2 + x^2y)}{1-x}$ term in the functional equation, where the $1/(1-x)$ accounts for the points attached to the boxes corresponding to the vertices in $M$.  An example of this case
		is shown on the right in Figure~\ref{fig:Qproof}.
		\end{enumerate}
	The justification for the equation given for $Q_{down}$ is a symmetry of the arguments
	given for $Q_{up}$.
\end{proof}

\begin{figure}[ht]
\begin{center}
\begin{tikzpicture}[scale=\picscales]
    \thepoints[(0,3)]{4}
    \thepoints[(5,4)]{1}
    \thegridext{4}
    \shadetheboxes[(0,4)]{1/5, 2/5, 3/5, 4/5}
    \begin{scope}[rotate=-90, yscale=-1, shift={(0,-7)}]
        \thepoints[(1,1)]{2}
        \thepoints[(4,3)]{1}
        \thegridext{3}
    \end{scope}
    \begin{scope}[shift={(6,-5)}]
        \thepoints[(1,0)]{1}
        \thepoints[(2,2)]{1}
        \thegrid{2}
    \end{scope}

    \begin{scope}[shift={(11,0)}]
        \thepoints[(0,3)]{4}
        \shadetheboxes[(0,4)]{1/5, 2/5, 4/5, 5/5, 6/5, 7/5, 8/5, 1/4, 2/4, 3/4, 4/4, 1/3, 2/3, 3/3}
        \labeltheboxes[(0,4)]{3/5/\permaslabel{0.2}{0.25}{1}/0/0}
    \thepoints[(5,4)]{1}
    \thegridext{4}
    \begin{scope}[rotate=-90, yscale=-1, shift={(0,-7)}]
        \thepoints[(1,1)]{2}
        \thepoints[(4,3)]{1}
        \thegridext{3}
    \end{scope}
    \begin{scope}[shift={(6,-5)}]
        \thepoints[(1,0)]{1}
        \thepoints[(2,2)]{1}
        \thegrid{2}
    \end{scope}
    \end{scope}
\end{tikzpicture}
\caption{On the left is an example of case~C1 in the proof of Lemma~\ref{lem:qind},
and on the right an example of case~C2.}
\label{fig:Qproof}
\end{center}
\end{figure}

Remark~\ref{rmk:skewclosure} implies we can find the generating function for
$(1234,1324,2143)$-avoiding permutations.

\begin{proposition} \label{prop:Q}
    Let $Q = Q(x,y)$ be the generating function where the coefficient of
    $x^n y^k$ is the number of independent sets of size $k$ in updown-cores on
    boundaries from $\Av_n(123,2143)$. Then
    \[
        Q = \frac{1}{1-Q_{ind}}.
    \]
    By setting $x = z$ and $y = z$ we get the generating function
    \eqref{eqn:nicegenfunc} for the number of $(1234,1324,2143)$-avoiding
    permutations.
\end{proposition}

The generating function gives the enumeration
\[
    1, 1, 2, 6, 21, 75, 268, 958, 3425, 12245, 43778, 156514, 559565, \dots
\]
and has been added to the Online Encyclopedia of Integer
Sequences~\cite[A263790]{oeis}.

\subsection{Updown-cores and the class $\av{1234,1324,1432,3214}$}
Here we replace avoidance of $2143$ with avoidance of $1432$ and $3214$ and
show that this subclass has the rational generating function

\begin{equation} \label{eqn:notquiteasnicegenfunc}
    \frac{1 - z - z^2 - z^3}{1 - 2z - z^2 - 2z^3 - 4z^4 - 8z^5 + 15z^7 + 14z^8 + 7z^9}.
\end{equation}

If a permutation $\pi$ avoids $1324$, then it contains $1432$ ($3214$) if
and only if $\bd{\pi}$ contains $1432$ ($3214$).
Proposition~\ref{prop:disjoint}~\eqref{prop:disjoint3} therefore implies the
equation
\[
    \av{1234, 1324, 1432, 3214} = \bigsqcup_{\pi \in \av{123, 1432, 3214}} \udperms{\bg{\pi}}.
\]

As before we need a structural description of the boundary grids indexed in the disjoint union
above:

\begin{lemma}\label{lem:ud14323214}
A boundary grid of a skew-indecomposable permutation in
\[
	\av{123,1432,3214} \backslash \{1,12,2143\}
\]
is an alternating sequence of
grids of the form $\bg{132}$ and $\bg{213}$ sharing their northwestern-most box
with the southeastern-most box of the next grid.
\end{lemma}

\begin{proof}
	Let $\pi$ be a skew-indecomposable permutation in $\av{123,1432,3214} \backslash \{1,12,2143\}$.
	The number of points northeast of a left-to-right minimum in $\pi$ is at most two, and likewise
	the number of points southwest of a right-to-left maximum in $\pi$ is at most two. We leave it
	to the reader to
	show that the boundary grid of $\pi$ has the claimed structure by following a similar argument
	as in the proof of Lemma~\ref{lem:qind}. See Figure~\ref{fig:next}.
\end{proof}

\begin{figure}[ht]
\begin{center}
\begin{tikzpicture}[scale=\picscales]
    \thepoints[(0,3)]{1}
    \thepoints[(-1,1)]{1}
    \thepoints[(2,2)]{1}
    \thepoints[(1,-1)]{1}
    \thepoints[(4,0)]{1}
    \thepoints[(3,-2)]{1}
    \thegrid{2}
    \begin{scope}[rotate=-90, yscale=-1, shift={(-1,-3)}]
        \thegrid{2}
    \end{scope}
    \begin{scope}[shift={(2,-2)}]
        \thegrid{2}
    \end{scope}
    \begin{scope}[rotate=-90, yscale=-1, shift={(-3,-1)}]
        \thegrid{2}
    \end{scope}

\end{tikzpicture}
\caption{The boundary grid of the permutation $462513$.}
\label{fig:next}
\end{center}
\end{figure}

Having the structural description of the boundary grids allows us to enumerate
the updown-cores by the size of the independent set.

\begin{lemma}
    Let $S_{ind} = S_{ind}(x,y)$ be the generating function where the
    coefficient of $x^n y^k$ is the number of independent sets of size $k$ in
    updown-cores on boundaries of skew-indecomposable permutations in
    $\Av_n(123,1432,3214)$. Then
    \[
        S_{ind} = x + x^2(1 + y) + x^4(1 + 7 y + 7 y^2 ) + S_{up} + S_{down},
    \]
    where $S_{up} = S_{up}(x,y)$, $S_{down} = S_{down}(x,y)$ satisfy the
    equations
    \[
       S_{up} = x^3 (1 + 3 y + y^2 ) + x S_{down} + x y S_{down} + x^2 y (S_{up} + (1+y)x^2),
    \]
    \[
        S_{down} = x^3 (1 + 3 y + y^2 ) + x S_{up} + x y S_{up} + x^2 y (S_{down} + (1+y)x^2).
    \]
\end{lemma}

\begin{proof}
We define $S_{up}$ to be the generating
function for the skew-indecomposable grids whose right-most grid is of the form
$\bg{213}$ and $S_{down}$ for the skew-indecomposable grids whose right-most grid
is of the form $\bg{132}$. Therefore if either has only one part it will be
counted by $x^3 (1 + 3y + y^2)$. The derivation of the equations for $S_{up}$
and $S_{down}$ are analogous to the proof of Lemma~\ref{lem:qind} and are omitted.
%We first count $S_{up}$ by choosing a vertex in the bottom row of the right-most
%part of its skew-sum. We can contain at most one of these two vertices in an
%independent set. If we choose neither or the right vertex, the right-most boxes
%are defined by some skew-indecomposable permutation in $\av{123,1432,3214}$ with
%a skew-sum starting with a $\bg{132}$. If we choose the left vertex then the
%right-most boxes are defined by some skew-indecomposable permutation in
%$\av{123,1432,3214}$ with a skew-sum starting with a $\bg{213}$ or is a single
%box. This leads the generating function above. $S_{down}$ is derived
%analogously. Therefore $S_{ind}$ is as claimed where
We add in $x$ for the
single point, $x^2 (1+y)$ for $\bg{12}$ and $x^4 (1 + 7y + 7y^2)$ for
$\bg{2143}$.
\end{proof}

Remark~\ref{rmk:skewclosure} implies we can find the generating function for
$(1234,1324,1432,3214)$-avoiding permutations.

\begin{proposition}\label{prop:S}
    Let $S = S(x,y)$ be the generating function where the coefficient of
    $x^n y^k$ is the number of independent sets of size $k$ in updown-cores on
    boundaries from $\Av_n(123,1432,3214)$. Then
    \[
        S = \frac{1}{1-S_{ind}}.
    \]
    By setting both $x=z$ and $y=z$ we get the generating function
    \eqref{eqn:notquiteasnicegenfunc} for the number of
    $(1234,1324,1432,3214)$-avoiding permutations.
\end{proposition}
%\begin{proof}
%Taking the skew-sum of skew-indecomposable permutations in the set
%$\av{1234,1324,1432,3214}$ gives us the entire set $\av{1234,1324,1432,3214}$.
%\end{proof}

The generating function gives the enumeration
\[
    1, 1, 2, 6, 20, 62, 172, 471, 1337, 3846, 11030, 31442, 89470, 254934 \dots
\]
and has been added to the Online Encyclopedia of Integer
Sequences~\cite[A260696]{oeis}.

\section{Conclusions and future work}
\label{sec:future_work}
%!TEX root = patt-gons.tex

%Smallest pole for Ps comes from the denominator
%x = 0.22816

%Smallest pole comes from the denominator
%x = 0.279706560943202

% Smallest pole comes from the denominator
% x≈0.350689

% smallest root
% x≈0.245122333753307

\subsection*{Growth rates}

As already noted the enumeration for the number of permutations that avoid $1324$ is unknown,
and not even the growth rate, i.e., the limit
\[
	\lim_{n \to \infty} \sqrt[n]{\Av_n(1324)},
\]
is known. Recently Bevan et al.~\cite{gr} established that the growth rate is in the
interval $[10.271, 13.5]$. We hope that the methods developed here can lead to some
non-trivial lower bounds for this growth rate, even though the classes studied here did
not yield good bounds, see Table~\ref{tbl:grw}.

\begin{table}[!htbp]
{\small
\begin{tabular}{llc}
Class                           & Reference                             & $\approx$ growth rt.\\
\hline
$\av{1324,2143}$              & Prop.~\ref{prop:smooth}, eqn~\eqref{eqn:smoothgenfunc}     & 4.4\\
$\av{1234,1324,2143}$         & Prop.~\ref{prop:Q}, eqn.~\eqref{eqn:nicegenfunc}           & 3.6\\
$\av{1234; 1324; 1432; 3214}$ & Prop.~\ref{prop:S}, eqn.~\eqref{eqn:notquiteasnicegenfunc} & 2.9 \\
%Bd.~type $(a,2)$              & Prop.~\ref{prop:x2H},                                      & 4.1 \\
%Bd.~type $(a,3)$              & Prop.~\ref{prop:x3J},                                      & 4.1
\end{tabular}
}
\caption{The growth rates for the subclasses of $\av{1324}$ studied in this paper.}
\label{tbl:grw}
\end{table}

\subsection*{Generalized patterns}
\label{subs:Vincular patterns}
We have focused on classical pattern avoidance in this paper. It is likely that our methods
can be used to study the avoidance of generalized patterns. For example consider the sets
\begin{align*}
	\av{1\underline{23}4} &= \{ \pi \vert \pi \text{ avoids an occurrence of $1234$ with $2$ and $3$ adjacent}\} \\
	\av{1\underline{32}4} &= \{ \pi \vert \pi \text{ avoids an occurrence of $1324$ with $3$ and $2$ adjacent}\},
\end{align*}
which are defined by \emph{vincular} patterns, see Babson and Steingrimsson~\cite{Babson:2000tr}.
Our methods, give a bijection between the two sets, by first mapping to the boundary encoding,
and then back to permutations by reversing the points in every column.

\subsection*{Non-crossing and non-nesting partitions}
\label{sub:Non-crossing and non-nesting partitions}
It was noted by Galashin~\cite{galashin} that bump-diagrams, see
e.g.,~Rubey and Stump~\cite{rubeystump}, for non-crossing and non-nesting
partitions, provide an alternative framework for the independent sets in $132$-
and $123$-cores.
%: An independent set in a $132$-core ($123$-core) corresponds to
%a non-crossing (non-nesting) partition; see Figure~\ref{fig:partition}.
%\begin{figure}[htp]
%\begin{center}
%\begin{tikzpicture}[scale=\picscale*0.6]
%
%    \thecoreAsmall{1}
%
%    \thecoreAsmall[(3,-1)]{2}
%
%    \thecoreAsmall[(7,-2)]{3}
%
%    \thecoreAsmall[(12,-3)]{4}
%
%    \thecoreAsmall[(18,-4)]{5}
%
%    \begin{scope}[shift={(0,-7.5)}]
%        \foreach \i in {2,...,6}{
%        \pgfmathsetmacro\ii{\i+1}
%        \pgfmathsetmacro\iii{\i}
%        \setpartitions[(\ii*\iii/2-3,0)]{\i}
%        }
%    \end{scope}
%
%    \begin{scope}[shift={(0,-10)}]
%        \thecoreBsmall{1}
%
%        \thecoreBsmall[(3,-1)]{2}
%
%        \thecoreBsmall[(7,-2)]{3}
%
%        \thecoreBsmall[(12,-3)]{4}
%
%        \thecoreBsmall[(18,-4)]{5}
%    \end{scope}
%
%\end{tikzpicture}
%
%\caption{The first (last) line shows the $132$-cores ($123$-cores) of sizes
%$1, \dots, 5$. In between is the amalgamation of the bump-diagrams of all
%set partitions of $\{1, \dots, n+1\}$, where $n$ is the size of the
%corresponding core.}
%\label{fig:partition}
%\end{center}
%\end{figure}

%We note that the independent sets in the $132$-core are the vertices of the
%non-crossing complex and the independent sets in the $123$-core are the vertices
%of the non-nesting complex, see e.g.,~Santos et al.~\cite{complexes}. Both of
%these complexes are pure.
From empirical testing, it seems that the down-cores
on boundaries avoiding $123$ are pure if and only if the boundary avoids $2143$.
We record this as a conjecture:

\begin{conjecture}
    Let $\pi$ be a $123$-avoiding permutation. The independent set complex of
    the down-core $D(\bg{\pi})$ is pure if and only if $\pi$ avoids $2143$.
\end{conjecture}

\subsection*{Acknowledgements}
We thank \'{E}mile Nadeau and the anonymous referees for thoughtful comments that helped
with clarifying the presentation of these results.

\bibliographystyle{amsalpha}
\bibliography{patt-gons}

\end{document}